\documentclass[10pt, a4paper]{article}

\usepackage{authblk}

\usepackage{bm}
\usepackage{subfigure}
\usepackage{indentfirst}
\usepackage{latexsym}
\usepackage{amsfonts}
\usepackage{amsmath,amssymb,amsfonts,mathrsfs, amsthm, color}
\usepackage{cite}
\usepackage{tikz}
\usepackage{caption}
\usepackage{graphicx}
\usepackage{epstopdf}
\usepackage[all]{xy}
\usepackage{mathrsfs}

\usepackage[top=2.5cm,bottom=3cm,left=2.5cm,right=2cm]{geometry}
\newcommand{\upcite}[1]{\textsuperscript{\textsuperscript{\cite{#1}}}}
\newtheorem{definition}{Definition}[section]
\newtheorem{Lemma}[definition]{Lemma}

\newtheorem{theorem}[definition]{Theorem}
\newtheorem{proposition}[definition]{Proposition}
\newtheorem{eg}[definition]{Example}
\newtheorem{example}[definition]{Example}

\newtheorem{rk}[definition]{Remark}
\newtheorem*{ack}{Acknowledgement}
\usepackage{amsmath}

\makeatletter
\newcommand{\rmnum}[1]{\romannumeral #1}
\newcommand{\Rmnum}[1]{\expandafter\@slowromancap\romannumeral #1@}
\makeatother

\title{Writhe polynomial for virtual links}

\author[1]{Mengjian Xu \thanks{xmjmath@mail.bnu.edu.cn}}
\affil[1]{School of Mathematical Sciences, Beijing Normal University, Beijing 100875, China; }
\affil[1]{Laboratory of Mathematics and Complex Systems, Ministry of Education, Beijing 100875, China}

\date{}

\begin{document}
\maketitle
\begin{abstract}

A weak chord index $Ind'$ is constructed for self crossing points of virtual links. Then a new writhe polynomial $W$ of virtual links is defined by using $Ind'$. $W$ is a generalization of writhe polynomial defined in \cite{Z.Cheng2013}. Based on $W$, three invariants of virtual links are constructed. These invariants can be used  to detect the non-trivialities of Kishino knot and flat Kishino knot.

\end{abstract}

\section{Introduction}\label{Introduction}

A knot is a 1-sphere $S^{1}$ embedded in 3-sphere $S^{3}$. One way to describe a knot is using knot diagram, which is a 4-valence graph obtained by the projection $S^{3}\rightarrow S^{2}$ with a general position. Every knot diagram can be represented by a unique Gauss diagram (see Definition \ref{Gauss diagram}). But there are Gauss diagrams that can not be realized by knots. Kauffman generalized knot theory to virtual knot theory in \cite{Kauffman.1998} to overcome this defect of classical knot theory. I will discuss this generalization in Section \ref{Poly} in more details. Introductory references \cite{Kauffman.2012}, \cite{Vassily} and \cite{Manturov2013} of virtual knot theory are recommended to readers. There is another interpretation of virtual knot theory. Roughly speaking, virtual knots are equivalent isotopic classes of knots in thickened surfaces modulo stablization (see \cite{J.Scott}). Kuperberg gave a topological interpretation of virtual links in \cite{what}.

A natural question is whether we can determine a virtual link is classical or not. The main step for solving this question is to look for computable and efficient invariants of virtual links. There are many invariants of virtual links, which are generalizations of invariants of classical links, such as quandle \cite{Manturov2002}, Jones polynomial, Khovanov homology\cite{Manturov2007}.  We want to seek for some other invariants of virtual links having high efficiency in telling a link is classical or not.  Considering this, there is a class of invariants having this property, which has the same value on classical links, such as Alexander polynomial\cite{Sliver}, span of link \cite{Z.Cheng2013}. However, if you confine to virtual knot, the first invariant to notice is the odd writhe\cite{Kauffman.2004}. Many researchers generalized it to a polynomial invariant $W_{K}(t)=\sum_{Ind(c)\neq 0}w(c)t^{Ind(c)}$ independently in \cite{Z.Cheng2013}, \cite{Dye13}, \cite{Im13}, \cite{Kauffman.2013} and \cite{Satoh14}. Following \cite{Z.Cheng2013}, I will call this polynomial invariant writhe polynomial in this paper.

Writhe polynomial has many advantages in detecting virtuality. For example, it is easy to compute and efficient to determine the virtuality of a large class of virtual knots. It also gives a constraint on periodicity of virtual knot (see \cite{Bae}). Besides, it can be used to construct flat invariants easily (see \cite{Im13}). It has many variants such as zero polynomial  \cite{Jeong}, transcendental polynomial  \cite{Z.Cheng.2015}, L-polynomial  \cite{Kaur}, etc. However, most  variants of writhe polynomial including itself are invariants of virtual knots but not virtual links. The main content of this paper is to define a link version of writhe polynomial and discuss some applications of it.

There is a link version of writhe polynomial, which appeared in \cite{Im17}. Let me briefly review it. Suppose $L=K_{1}\cup \cdots \cup K_{n}$ is an $n$-component virtual link diagram. The writhe polynomial in the sense of \cite{Im17} is $$Q_{L}(x,y)=\sum_{c\in S(L)}\omega(c)(x^{i(c)}-1)+\sum_{c\in M(L)}\omega(c)(y^{i(c)}-1),$$ where $S(L)$ and $M(L)$ are the set of self crossing points and the set of linking crossing points respectively. $i(c)$ is called virtual intersection index in \cite{Lee14}. For more details of terminologies in this equation, see \cite{Im17} and \cite{Lee14}. The first sum of $Q_{L}(x,y)$ is just $\sum_{i=1}^{n}P_{K_{i}}(x)$. Here $P_{K_{i}}(x)$ denotes the affine index polynomial of \cite{Kauffman.2013}, which is equivalent to writhe polynomial. For $c\in M(K_{r}\cup K_{j})$, $|i(c)|$ is the span of $K_{r}\cup K_{j}$, where $r$ is different from $j$. It is easy to see that there is no interrelationship between self crossings and linking crossings in this polynomial.

Before I advertise that my polynomial include some information of internal relation between self crossing points and linking crossing points, let us take a closer look at writhe polynomial $W_{K}(t)$. The main constituent of writhe polynomial is the index function $Ind(c)$.  Cheng axiomatized $Ind(c)$ in \cite{Z.Cheng.2017}. These axioms are called  chord index axioms (see Definition \ref{axiom}). If you have a chord index function in hand, then you can formally construct an invariant of virtual links. For details, see Section \ref{Poly}.

My strategy for constructing writhe polynomial of virtual links is to construct an new index function denoted by $Ind'(c)$, which satisfys all chord axioms of Definition \ref{axiom} except (1). There is indeed interrelationship between self crossing points and linking crossing points in $Ind'(c)$ (see Example \ref{eg1}). Then writhe polynomial is defined by $$W_{L}=\sum_{\substack{c\in S(L) \\ Ind(c)\neq 0}}\omega(c)t^{Ind'(c)}.$$ For more details, see Section \ref{Writhe poly}. At this moment, you may be unsatisfied with the sum is taken on self crossing points. In fact, Cheng-Gao defined the so-called linking polynomial $L(t)$ in \cite{Z.Cheng2013}, whose definition relies on "chord index" $N(c)$ of a linking crossing point $c$. Technically, $N(c)$ is not a chord index in the sense of Definition \ref{axiom}. They disposed some subtle nondeterminacy. The polynomial invariant $Q_{L}(x,y)$ in \cite{Im17} can be regarded as combining the informations of writhe polynomial and $L(t)$. To some extent, $W_{L}(t)$ and $L(t)$ are complementary.

Let me finish this section with an introduction to the structure of this paper. I recall some classical results of virtual knot theory in Section \ref{Poly}, develop fundamental tools of this paper in Section \ref{Axioms_index} and introduce some basic flat virtual knot theory in Section \ref{flat virtual knot}. Then the writhe polynomials $W_{L}(t)$ is constructed in Section \ref{Writhe poly}. Combining $W_{L}(t)$ with invariant $\mathcal{L}(L)$ of \cite{Z.Cheng.2018}, three invariants are constructed in Section \ref{Kishino}. And they can be used to detect the non-trivialities of Kishino knot and its variant. Finally, I prove a technical lemma i.e. Lemma \ref{wci} in Section \ref{tech}.

\section{Virtual knot theory and polynomial invariants of virtual knots }\label{Poly}

Classical knot theory studies the embeddings of $S^{1}\rightarrow S^{3}$ up to isotopy.  From combinatorial viewpoint, classical links are equivalent classes of link diagrams modulo Reidemeister moves i.e. the first three moves of  Fig.\ref{GRM}. The virtual knot theory can be seen as a combinatorial generalization of classical knot theory. A \emph{virtual link diagram} is a 4-valence graph, part of whose vertices are crossing points. The rest vertices are called virtual crossing points. And we use a circle at the virtual crossing point to point out it's virtuality (see Fig.\ref{VCP}). \emph{Virtual links} are equivalent classes of virtual link diagrams modulo generalized Reidemeister moves i.e. moves of Fig.\ref{GRM}.

\begin{figure}
\centering
 \subfigure{\includegraphics[height = 2cm, width = 2cm]{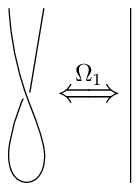}}
 \ \ \ \ \ \
 \subfigure{\includegraphics[height = 2cm, width = 2cm]{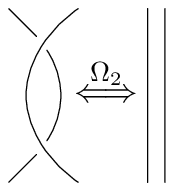}}
 \ \ \ \ \ \
 \subfigure{\includegraphics[height = 2cm, width = 4cm]{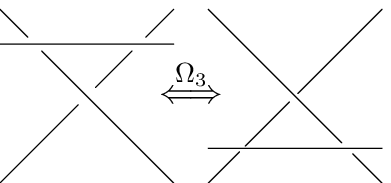}}

 \subfigure{\includegraphics[height = 2cm, width = 2cm]{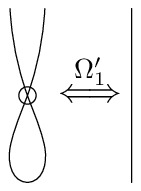}}
 \ \ \ \ \ \
 \subfigure{\includegraphics[height = 2cm, width = 2cm]{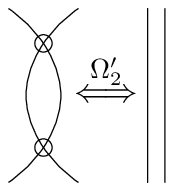}}
 \ \ \ \ \ \
 \subfigure{\includegraphics[height = 2cm, width = 4cm]{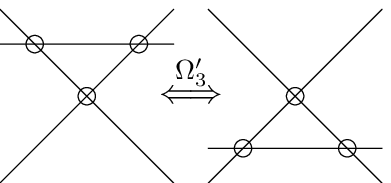}}
 \ \ \ \ \ \
 \subfigure{\includegraphics[height = 2cm, width = 4cm]{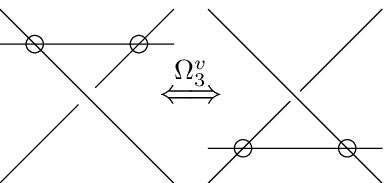}}\caption{generalised Reidemeister moves}\label{GRM}

\end{figure}

In Fig.\ref{GRM}, the first three moves are called Reidemeister moves. The next three are called virutal Reidemeister moves. The last one is called semi-virtual Reidemeister moves.

\begin{figure}
 \centering
 \includegraphics[height =1cm, width =1cm]{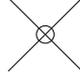}
 \caption{virtual crossing point}\label{VCP}
\end{figure}

\begin{definition}\upcite{Vassily}\label{Gauss diagram}
Given a virtual knot diagram, the Gauss diagram consists of a circle representing the preimage
of the virtual knot, a chord on the circle connecting the two preimages of each crossing point, an arrow
on each chord directed from the preimage of the overcrossing to the preimage of the undercrossing, a
positive or negative sign on each chord representing the writhe(Fig.\ref{writhe}) of that crossing point.

\end{definition}

\begin{figure}
 \centering
 \includegraphics[height = 1cm, width = 3cm]{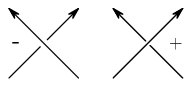}
 \caption{writhe}\label{writhe}
\end{figure}

It is well known that Gauss diagram of virtual knot diagram completely determines the virtual knot diagram up to virtual and semi-virtual Reidemeister moves.

We can also define \emph{Gauss diagram} of a virtual link diagram $L=K_{1}\cup\cdots \cup K_{n}$. The Gauss diagram consists of $n$ Gauss diagrams of $K_{1}$, $K_{2}$, $\cdots$, and $K_{n}$, a chord between two Gauss diagrams corresponding to a linking crossing point, an arrow and an sign on each chord as in Definition \ref{Gauss diagram}. I want to emphasize that every circle in the Gauss diagram of $L$ has counterclockwise orientation.

There are some notation conventions.  In this paper, I use $K$ ($L$) to denote a knot (link) diagram, the corresponding knot (link) class and the corresponding Gauss diagram. And I use $c$ to denote a real crossing  and the corresponding chord. The exact meaning of these notations can be read from the context. There are two kinds of real crossing point in $L$. A real crossing point of a single component is called a self crossing point. The rest of real crossing points of $L$ are called linking crossing points. Let $S(L)$, $M(L)$ and $C(L)$ denote the set of self crossing points, the set of linking crossing points and the set of all real crossing points of $L$ respectively. 

Let us recall the definition of index in \cite{Im172}. For a chord $c$ of a Gauss diagram $L=K_{1}\cup \cdots \cup K_{n}$, we assign a sign for each endpoint of $c$ as shown in Figure \ref{endpt}, where $e$ represents an endpoint of $c$. If $c$ is a crossing point of $K_{i}$, it divides $K_{i}$ into two parts, where $K_{i}$ also represents the corresponding components of the Gauss diagram $L$. The left part of $c$ is shown in Figure \ref{left}, where $e$ represents some endpoint of some other chord. We use $left(c)$ to denote the left part of $c$. Then the index function of $c$ is defined as follows.


\begin{figure}
 \centering
 \includegraphics[height =2cm, width =10cm]{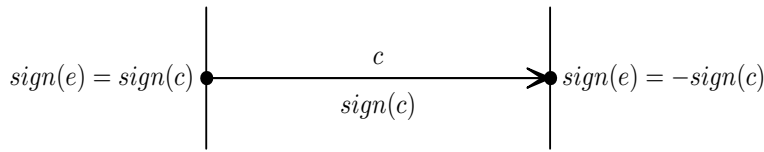}
 \caption{sign of endpoints}\label{endpt}
\end{figure}

\begin{figure}
 \centering
 \includegraphics[height =2cm, width =3cm]{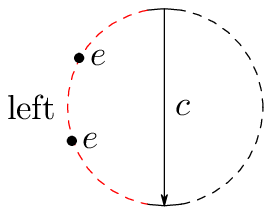}
 \caption{left part}\label{left}
\end{figure}

\begin{definition}\label{Ind for knot}

Given a Gauss diagram  $K$, we define index function for every chord $c$ as follows:

$$Ind(c)=\sum_{e\in left(c)}sign(e),$$ $e$ is an endpoint of some chord other than $c$.

\end{definition}

 Cheng used this index function to define odd writhe polynomial in \cite{Z.Cheng.2012} as follows: $$f_{K}(t)=\sum_{Ind(c_{i})\neq 0 \ \ mod 2}\omega(c_{i})t^{Ind(c_{i})+1},$$ where $\omega(c)$ denotes the writhe of $c$. Then, Cheng and Gao generalized it to writhe polynomial\cite{Z.Cheng2013} as follows: $$W_{K}(t)=\sum_{Ind(c)\neq 0}\omega(c)t^{Ind(c)}. $$ Here $W_{K}$ is slightly different from but equivalent to the original one. This is closely related to affine index polynomial defined by Kauffman in \cite{Kauffman.2013}, which is $$P_{K}(t)=\sum_{c\in C(K)}\omega(c)t^{Ind(c)}-\omega(K).$$  $\omega(K)$ denotes the writhe of the knot diagram $K$. This polynomial invariant was independently introduced in \cite{Dye13}, \cite{Im13} and \cite{Satoh14}.


 \section{Axioms of index function}\label{Axioms_index}

 As I mentioned in the introduction to this paper, the main constituent of $W_{K}(t)$ and $P_{K}(t)$ is the index function $Ind$. Many authors modify the index function $Ind$ to obtain variants of writhe polynomial. Cheng axiomatized the index function as follows.

\begin{definition}\upcite{Z.Cheng.2018}\label{axiom}
Assume for each real crossing point $c$ of a virtual link diagram, according to some rules we can assign an index (e.g. an integer, a polynomial, a group etc.) to it. We say this index satisfies the chord index axioms if it satisfies the following requirements:

(1) The real crossing point involved in $\Omega_{1}$ has a fixed index (with respect to a fixed virtual link);

(2) The two real crossing points involved in $\Omega_{2}$ have the same indices;

(3) The indices of the three real crossing points involved in $\Omega_{3}$ are preserved under $\Omega_{3}$ respectively;

(4) The index of the real crossing point involved in $\Omega_{3}^{v}$ is preserved under $\Omega_{3}^{v}$;

(5) The index of any real crossing point not involved in a generalized Reidemeister move is preserved under this move.

\end{definition} Let us call an index function a chord index, if it satisfies chord index axioms.

If you have a chord index $ind$ in hand, you can construct a formal polynomial invariant of virtual links as follows. $$F_{L}(t)=\sum_{c}\omega(c)(t^{ind(c)}-t^{f(L)}),$$  where $f(L)$ is the fixed index with respect to $L$ in the first axiom.

To prove the invariance of $F_{L}(t)$, the next lemma is very useful. It will also be used in this paper. This lemma is due to Polyak.

 \begin{Lemma}\upcite{MP}\label{Polyak}
There is an minimal generating set of  Reidemeister moves shown in Figure \ref{Polyakresults}.
\end{Lemma}

Here Reidemeister moves mean the oriented version of the three moves shown in Figure \ref{GRM}.

\begin{figure}

\centering
 \subfigure[$\Omega1a$]{\includegraphics[height = 2cm, width = 2cm]{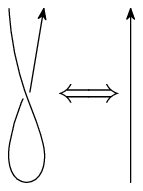}}
 \ \ \ \ \ \
 \subfigure[$\Omega1b$]{\includegraphics[height = 2cm, width = 2cm]{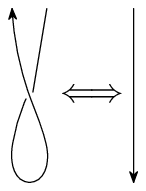}}
 \ \ \ \ \ \
 \subfigure[$\Omega2a$]{\includegraphics[height = 2cm, width = 2cm]{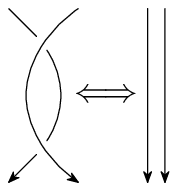}}
 \ \ \ \ \ \
 \subfigure[$\Omega3a$]{\includegraphics[height = 2cm, width = 4cm]{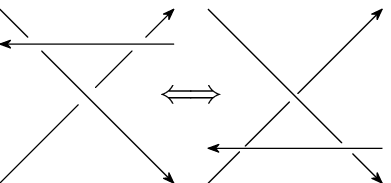}}\caption{generators of classical Reidemeister moves}\label{Polyakresults}
\end{figure}

This lemma can reduce the cases of Reidemeister moves in the proof. Hence, we can modify the chord index axioms as follows.


\begin{definition}
An index function for self crossing points is called a weak chord index if it satisfies the following axioms:

(2') The two self crossing points involved in $\Omega_{2a}$ have the same indices;

(3') The indices of the self crossing points involved in $\Omega_{3a}$ are preserved under $\Omega_{3a}$;

(4) and (5) of chord index axioms, where real crossing points are replaced by self crossing points.

Here $\Omega_{2a}$ and $\Omega_{3a}$ are shown in Figure \ref{Polyakresults}.
\end{definition}

Inspired by $W_{K}(t)$, we can construct an invariant as follows.

\begin{theorem}\label{model}
$$G_{L}(t)=\sum_{\substack{c\in S(L) \\ Ind(c)\neq 0}}\omega(c)t^{wci(c)}$$ is an invariant, where $wci$ is a weak chord index.

\end{theorem}

\begin{proof}
Suppose link diagram $L'$ is obtained by conducting a Reidemeister move $\Omega$ on $L$.


If $\Omega=\Omega_{1}$ and $c$ is the crossing point involved in $\Omega_{1}$, then  $Ind(c)=0$. So, $c$ has no contribution to $G_{L}(t)$. On the other hand, there is no corresponding crossing point in $L'$.

If $\Omega=\Omega_{2a}$ and the crossing points involved in $\Omega_{2a}$ are denoted by $c_{1}$ and $c_{2}$, then either both $Ind(c_{1})$ and $Ind(c_{2})$ are $0$ or neither of them are by (2). Hence the contributions of $c_{1}$ and $c_{2}$ to $G_{L}(t)$ cancel out by (2'). On the other hand, there are no corresponding crossing points in $L'$.


Suppose $c$ is a self crossing point satisfying one of following condition:

  \begin{itemize}

\item $\Omega=\Omega_{3a}$ and $c$ is a arbitrary self crossing point involved in $\Omega_{3a}$;

\item $\Omega=\Omega_{3}^{v}$ and $c$ is the self crossing point involved in $\Omega_{3}^{v}$;

\item $c$ is a crossing point not involved in $\Omega$.

\end{itemize}

Let $c'$ be the corresponding crossing point of $c$ on $L'$. Then either both $Ind(c')$ and $Ind(c)$ are $0$ or neither of them are by (3), (4) or (5). Hence $c$ and $c'$ have the same contribution to $G_{L}(t)$ and $G_{L'}(t)$ respectively.

We have verified the invariance on generators of generalized Reidemeister moves by Lemma \ref{Polyak}. Hence, $G_{L}(t)$ is an invariant.

\end{proof}

\begin{rk}

$Ind(c)\neq 0$ can be replaced by $ind(c)\neq f(L)$, here $ind(c)$ is any index function who satisfies the chord index axioms and $f(L)$ is the fixed index with respect to $L$ in the first axiom. The proof is the same. For example, we can use chord index $g_{c}(s)$ of \cite{Z.Cheng.2015} to get a refined invariant. But $Ind(c)\neq 0$ is enough for proving the virtuality of all examples of this paper especially Kishino knot.

\end{rk}

\section{Flat virtual knot theory}\label{flat virtual knot}

Before defining the writhe polynomial of virtual links, let me recall some basic facts of flat virtual knot theory. I will test our theory on some flat virtual knots in Section \ref{Kishino}.
\begin{definition}
A flat virtual link diagram is a virtual link diagram with the classical points replaced by flat points, which are points without overcrossing and undercrossing  information. \emph{Flat links} are equivalence classes of link diagrams with equivalence relation generated by flat Reidemeister moves(Fig.\ref{freeGRM}).
\end{definition}

\begin{figure}
\centering
 \subfigure{\includegraphics[height = 2cm, width = 2cm]{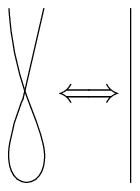}}
 \ \ \ \ \ \
 \subfigure{\includegraphics[height = 2cm, width = 2cm]{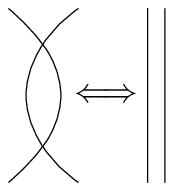}}
 \ \ \ \ \ \
 \subfigure{\includegraphics[height = 2cm, width = 4cm]{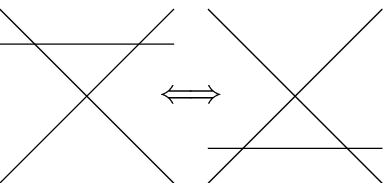}}

 \subfigure{\includegraphics[height = 2cm, width = 2cm]{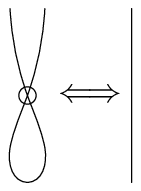}}
 \ \ \ \ \ \
 \subfigure{\includegraphics[height = 2cm, width = 2cm]{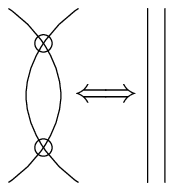}}
 \ \ \ \ \ \
 \subfigure{\includegraphics[height = 2cm, width = 4cm]{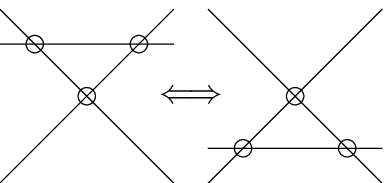}}
 \ \ \ \ \ \
 \subfigure{\includegraphics[height = 2cm, width = 4cm]{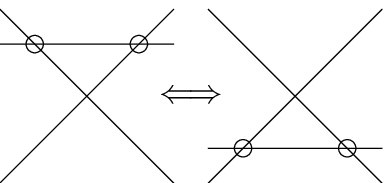}}\caption{Flat Reidemeister moves}\label{freeGRM}

\end{figure}

There is an obvious surjection $F: \{virutal\ links\}\rightarrow \{flat\ virtual\ links\}$. Flat virtual knot diagram $F(K)$ is obtained by replacing all real crossing points of $K$ with corresponding flat crossing points. Following \cite{Z.Cheng.2018}, $F(K)$ is also called the \emph{shadow} of $K$ and denoted by $\widetilde{K}$.  Flat virtual links can also be seen as equivalent classes of virtual links modulo crossing changes, where crossing changes mean switching real crossing points. We say a virtual link invariant $f$ is a flat invariant if it can factor through flat virtual links i.e.
\begin{center}
\
\xymatrix{
  \{virtual\ links\} \ar[rr]^{f} \ar[dr]_{}
                &  &  range       \\
                & \{flat\ virtual\ links\} \ar[ur]_{}              }
\end{center}

We can use writhe polynomial $W_{K}$ to construct a flat invariant of virtual knot $K$ as follows.
\begin{proposition}\label{ftl}
$\overline{W}_{K}$ is a flat invariant of virtual knots, where $\overline{W}_{K}=W_{K}(t)-W_{K}(t^{-1}).$

\end{proposition}

\begin{proof}

It suffices to prove $\overline{W}_{K}$ remains unchanged under crossing changes. Let $K'$ denote the virtual knot diagram obtained by performing crossing change on $c$. The Gauss diagram of $K'$ is obtained from the Gauss diagram of $K$ by reversing the orientation of arrow $c$, alternating the sign of it and remaining others unchanged. If $a\neq c$, $w(a)t^{Ind(a)}$ remains unchanged, but $w(c)t^{Ind(c)}$ becomes to $-w(c)t^{-Ind(c)}$. However, the contribution of $c$ to $\overline{W}_{K}(t)$ is $w(c)t^{Ind(c)}-w(c)t^{-Ind(c)}.$ Hence $\overline{W}_{K}$ remains unchanged under crossing changes.

\end{proof}

\begin{rk}\label{flatrk}

The coefficient of $t^{n}$ in $\overline{W}_{K}(t)$ is called n-th dwrithe in \cite{Kaur}, which is also a flat invariant.

In fact, Proposition \ref{ftl} follows from a general construction. Given a variant of writhe polynomial, say $F_{L}(t)$, then $\overline{F}_{L}(t)=F_{L}(t)+F_{\overline{L}}(t)$ is a flat invariant, where $\overline{L}$ is obtained form $L$ by switching all real crossing points of $L$. The proof is a direct computation.

\end{rk}

\section{Writhe polynomial of virtual links}\label{Writhe poly}

Now we are in position to give the precise definition of writhe polynomial for virtual links.
First, let us fix some notations. Let
$L=K_{1}\cup \cdots \cup K_{n}$ be a virtual link, and $c$ a self crossing point, which belongs to some single component $K_{i}$. Following \cite{Im17} and \cite{Lee14}, $Ind(c)$ should be computed by regarding $K_{i}$ as a single knot. We construct a new index function for $c$ as follows.
\begin{definition}\label{ind for link}

$$Ind'(c)=\sum_{e\in left(c)}sign(e).$$ $e$ is an endpoint of some chord other than $c$.

\end{definition}

Recall that $left(c)$ is the left part of $c$ shown in Figure \ref{left}. You may wonder whether Definition \ref{ind for link} is the same as Definition \ref{Ind for knot}. The difference between $Ind'(c)$ and $Ind(c)$ is $left(c)$ in Definition \ref{ind for link} would contain some endpoints of linking crossing points. This would lead to the following proposition.

\begin{proposition}\label{difference}

$Ind'$ does not satisfy the first axiom of Definition \ref{axiom} while $Ind$ does.

\end{proposition}

Before proving Proposition \ref{difference}, let me introduce the signed span for $K_{i}$, which will also be used later.

\begin{definition}

Let $\{ e_{1}, \cdots, e_{p}\}$ be the set of endpoints on $K_{i}$ such that $e_{j}$ is an endpoints of some linking crossing point for all $j\in \{1, \cdots, p\}$. The signed span of $K_{i}$ is defined by $$span_{K_{i}}=\sum_{j=1}^{p}sign(e_{j})$$

\end{definition}

\begin{rk}

The set $\{span_{K_{i}}| i=1, \cdots, n\}$ is an invariant of $L$. The proof is a direct verification of the invariance under Reidemeister moves.

\end{rk}

Proof of Proposition \ref{difference}:
\begin{proof}

It is easy to see $Ind(c)=0$ if $c$ is the real crossing point involved in $\Omega_{1}$.

In Figure \ref{noaxiom}, the left diagram is obtained from the right diagram by performing $\Omega_{1}$ and its inverse. However, $Ind'(c)=0$ while $Ind'(c')=span_{K_{i}}$.

\end{proof}

\begin{figure}
 \centering
 \includegraphics[height =2cm, width =4cm]{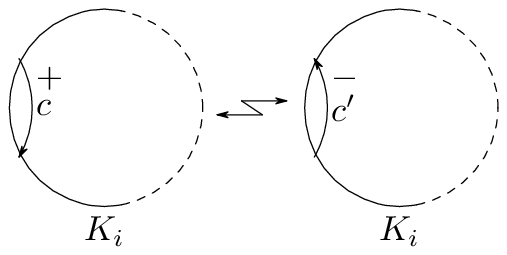}
 \caption{$Ind'$ does not satisfy (1)}\label{noaxiom}
\end{figure}

It is true that $Ind'(c)$ satisfies the rest of chord index axioms. The proof is not hard but troublesome. This is because Proposition \ref{difference} causes a trouble that we can not directly use Lemma \ref{Polyak}. It is easier to prove it satisfies the weak chord index axioms.

\begin{Lemma}\label{wci}
$Ind'$ is a weak chord index.

\end{Lemma}

The proof of this lemma is left to Section \ref{tech}, as it is tedious and unilluminating. Readers are recommended to partly verify this lemma by themselves.


Thus the writhe polynomial for virtual links is defined as follows:

\begin{definition}
$$Wr_{L}(t)=\sum_{\substack{c\in S(L) \\ Ind(c)\neq 0}}\omega(c)t^{Ind'(c)},$$ where $\omega(c)$ denotes the writhe of $c$.

\end{definition}
We temporarily use the notation $Wr$ to denote the writhe polynomial for virtual links while use $W$ to denote it after Proposition \ref{simplep}.

\begin{rk}
There is another approach to define writhe polynomial for virtual links. The proof of Proposition \ref{difference} suggests we can turn $Ind'$ into chord index by taking values in $\mathbb{Z}_{span_{L}}$. Here $span_{L}$ is the greatest common divisor of $|span_{K_{1}}|$, $\cdots$, $|span_{K_{n}}|$. Then we have writhe polynomial $Wr_{L}(t)=\sum_{c\in S(L)}\omega(c)(t^{Ind'(c)}-1)$. But there is a shortcoming of this approach. The polynomial is 0 if there is a pair of relatively prime integers $|span_{K_{i}}|$ and $|span_{K_{j}}|$. Specially, if some $|span_{K_{i}}|$ is 1, then the polynomial is 0. On the other hand, one can easily enhance the invariant into a polynomial in $t_{1}, \cdots, t_{n}$ if the virtual link $L$ is ordered.

\end{rk}

Next theorem is a direct consequence of Theorem \ref{model} and Lemma \ref{wci}.

\begin{theorem}\label{mainresult}
$Wr_{L}(t)$ is an invariant of virtual links.
\end{theorem}

The remainder of this section is to discuss some properties of $Wr_{L}(t)$.

\begin{proposition}\label{simplep}
(a) If $L$ is a classical link, then $Wr_{L}(t)=0$,

(b) If $L$ is a virtual knot $K$, then $Wr_{L}(t)=W_{K}(t)$,

(c) If $Wr_{L}(t)$ equals $\sum_{i=1}^{m}a_{i}t^{i}$, then $\#S(D)\geq \sum_{i=1}^{m}|a_{i}|$, $D$ is an arbitrary diagram representing $L$, $\#S(D)$ is the cardinal of set $S(D)$.

\end{proposition}
\begin{proof}
For (a), it is well-known every crossing point $c$ in a classical knot has $Ind(c)=0$. Then we have $Wr_{L}(t)=0$ for classical link.

(b), (c) can be seen from the definitions.

\end{proof}

Proposition \ref{simplep} implies that $Wr$ is a generalization of $W$. Hence, we use $W$ to denote the writhe polynomial instead of $Wr$.

The \emph{self crossing number} of $L$ is the minimum of number of self crossing points among diagrams representing $L$, which is also an invariant of $L$. Proposition \ref{simplep} shows $W_{L}(t)$ gives a lower bound of this number.

The next property has special status in this paper. I am about to write it exclusively.

For an $n$-component link $L=K_{1}\cup \cdots \cup K_{n}$, we can split the writhe polynomial $W_{L}(t)$ into $\sum_{i=1}^{n}W_{i}(t)$, where $W_{i}(t)=\sum_{\substack{Ind(c)\neq 0\\ c\in S(K_{i})}}\omega(c)t^{Ind'(c)}$. We set $$\overline{W}_{L}(t)=\sum_{i=1}^{n}\overline{W}_{i}(t),$$ where $\overline{W}_{i}(t)=W_{i}(t)-t^{span_{K_{i}}}W_{i}(t^{-1})$. When $L$ is a virtual knot, it is easy to see that $\overline{W}_{L}(t)$ is the same asp the polynomial in Proposition \ref{ftl}.

\begin{proposition}\label{flatweight}
$\overline{W}_{L}(t)$ is a flat invariant.

\end{proposition}
\begin{proof}
Let $L'=K'_{1}\cup \cdots \cup K'_{n}$ be a link diagram obtained by switching a crossing point $c$ on $L$. Let $c'$ denote the corresponding crossing point of $c$ on $L'$. For any self crossing point $a$ on $L$ different from $c$, the corresponding crossing point of $a$ on $L'$ is denoted by $a'$. Then the Gauss diagram of $L'$ is obtained by reversing orientation of chord $c$ and alternating the sign of chord $c$ on Gauss diagram of $L$. So, signs of endpoints of $c'$ are the same as those of endpoints of $c$.

 Then, $span_{K_{j}}=span_{K'_{j}}$ ($j\in \{1, \cdots, n\}$) and $Ind'(a)=Ind'(a')$. Hence, $a$ and $a'$ have the same contributions i.e. $\omega(a)t^{Ind'(a)}-\omega(a)t^{span_{K_{i}}-Ind'(a)}$ to $\overline{W}_{L}(t)$ and $\overline{W}_{L'}(t)$ respectively.

Suppose $c$ is a linking crossing point, $c$ and $c'$ have no contribution to $\overline{W}_{L}(t)$ and $\overline{W}_{L'}(t)$ respectively by the definition of $W$.

 Suppose $c$ is a self crossing point on $K_{i}$. The contributions of self crossing points to $Ind'(c)$ and the contributions of linking crossing points to $Ind'(c)$ are denoted by $A$ and $B$ respectively. $Ind'(c')$ also has a similar division $Ind'(c')=A'+B'$.

  Let $\{e_{1}, \cdots, e_{p}, e_{p+1}, \cdots, e_{p+q}\}$ be the set of endpoints on $K_{i}$, where $e_{j}$ is an endpoint of some linking crossing point for $j\in\{1, \cdots, p\}$ and an endpoint of some self crossing point for $j\in \{p+1, \cdots, p+q\}$.

We have following equations: $$B+B'=\sum_{j=1}^{p}sign(e_{j})=span_{K_{i}},$$  $$A+A'=\sum_{j=p+1}^{p+q}sign(e_{j})=0.$$ So, $Ind'(c')+Ind'(c)=span_{K_{i}}$. The contribution of $c$ to $\overline{W}_{L}(t)$ is $\omega(c)t^{Ind'(c)}-\omega(c)t^{span_{K_{i}}-Ind'(c)}=-\omega(c')t^{span_{K_{i}}-Ind'(c)}+\omega(c')t^{Ind'(c')}$, which is also the contribution of $c'$ to $\overline{W}_{L'}(t)$.

\end{proof}

The following proposition is a byproduct of above proof.

\begin{proposition}

$W_{\overline{L}}(t)=\sum_{i=1}^{n}-t^{span_{K_{i}}}W_{i}(t^{-1}) $, where $\overline{L}$ is obtained from $L$ by switching all real crossing points.

\end{proposition}

Let us end this section by an example.

\begin{example}\label{eg1}
\begin{figure}
\centering
 \subfigure{\includegraphics[height = 2cm, width = 3cm]{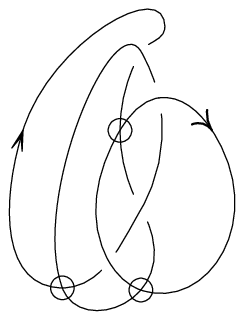}}
\ \
 \subfigure{\includegraphics[height = 2cm, width = 2cm]{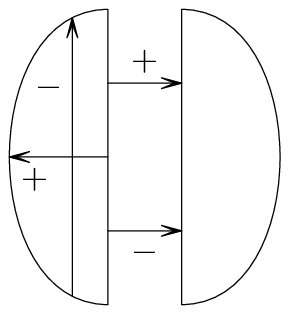}} \caption{Example \ref{eg1}}\label{fig_of_eg1}

\end{figure}
Let $L$ be the virtual link depicted in Figure \ref{fig_of_eg1}. Then $W_{L}(t)$ equals $t^{2}-t$ while $Q_{L}(x,y)$ equals 0. This shows that $W_{L}(t)$ indeed includes some new information of virtual links.  The self crossing number of $L$ is 2 by (c) of Proposition \ref{simplep}.

Moreover, $\overline{W}_{L}(t)$ equals $t^{2}-t+t^{-1}-t^{-2}$. This shows the shadow $F(L)$ of $L$ is nontrivial by Proposition \ref{flatweight}.

\end{example}

\section{Some polynomials of virtual links arising from $W_{L}$}\label{Kishino}

Proposition \ref{simplep} suggests that the writhe polynomial we defined coincides with the writhe polynomial defined in \cite{Z.Cheng2013} for any virtual knot $K$. It seems we can not use $W_{L}$ to get new information of virtual knots. However, combining with results of \cite{Z.Cheng.2018}, we can use $W_{L}$ to construct two transcendental polynomial invariants $\mathcal{L}_{L}(t,s)$ and $B_{L}(t,s)$ of virtual links.  Moreover, we can use $B_{L}(t,s)$ to construct a flat invariant $\overline{B}_{L}(t,s)$. Then we can use $B_{K}(t,s)$ to give a new proof of virtuality of Kishino knot (Figure \ref{Kishinoknot}), which can not be detected by any variant of $W(t)$ including itself.

\begin{figure}
 \centering
 \includegraphics[height =2cm, width =8cm]{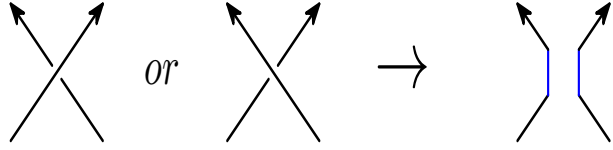}
 \caption{1-smoothing}\label{1smoothing}
\end{figure}

First, let us recall the invariant $\mathcal{L}(L)$, which appeared in \cite{Z.Cheng.2018}. $L$ is an $n$-component virtual link.  Let $\mathcal{M}$ be the free $\mathbb{Z}$-module generated by all oriented flat virtual links. $L_{c}$ is a virtual link obtained from $L$ by performing 1-smoothing at $c$ (Figure \ref{1smoothing}). Here $c$ is a real crossing point of $L$. $\widetilde{L}_{c}$ is the shadow of $L_{c}$.

\begin{theorem}\upcite{Z.Cheng.2018}\label{Lpoly}
$\mathcal{L}(L)=\sum_{c}\omega(c)\widetilde{L}_{c}-\omega(L)(\widetilde{L}\cup \mathcal{U})$ defines an $\mathcal{M}$-valued invariant of virtual links, where $\omega(L)$ is the sum of writhes of all crossing points. $\mathcal{U}$ is a trivial component of flat virtual link i.e. an unknot, which has no intersections with other components.

\end{theorem}

\begin{rk}\label{Kishinork}
The statement of Theorem \ref{Lpoly} was not spelled out in \cite{Z.Cheng.2018}. It was suggested by Remark 3.15 of \cite{Z.Cheng.2018}.

 Just as (c) of Proposition \ref{simplep}, we can use $\mathcal{L}(L)$ to obtain a lower bound of real crossing number (i.e. the minimum of number of real crossing points among diagrams representing $L$). And it is easy to see $\mathcal{L}(K)=0$ for classical knot $K$ for switching crossing point is an unlinking and unknotting operation.

\end{rk}


\cite{Z.Cheng.2018} proved that $\widetilde{L}_{c}$ is a chord index. And $W$ provides us a flat invariant $\overline{W}_{L_{c}}(t)$ and hence an invariant of $\widetilde{L}_{c}$. So, $\overline{W}_{L_{c}}(t)$ can be seen as an index function. It automatically satisfies chord index axioms except (1). In fact, it does satisfy (1).

\begin{Lemma}\label{flatW}

$\overline{W}_{L_{c}}(t)$ is a chord index.

\end{Lemma}

\begin{proof}

It suffices to verify $\overline{W}_{L_{c}}(t)$ satisfies the axiom (1).
If $c$ is a crossing point involved in $\Omega_{1}$, then $L_{c}=L\cup \mathcal{U}$ in the sense of virtual isotopic class. Hence $\overline{W}_{L_{c}}(t)=\overline{W}_{L}(t)$ by definition of $\overline{W}$.

\end{proof}

 Therefore we have following theorem.

\begin{theorem}

$\mathcal{L}_{L}(t,s)=\sum_{c}\omega(c)t^{\overline{W}_{L_{c}}(s)}-\omega(L)t^{\overline{W}_{L}(s)}$ is an invariant of virtual links.

\end{theorem}

$\mathcal{L}(L)$ is more like a theoretical framework than for practical use. You can see $\mathcal{L}_{L}(t,s)$ as an implementation of this framework. Hence $\mathcal{L}_{L}(t,s)$ has some properties of $\mathcal{L}(L)$. For example, $\mathcal{L}_{\overline{L}}(t,s)=-\mathcal{L}_{L}(t,s)$. However, $\mathcal{L}_{L}(t,s)$ does not have the next property of $\mathcal{L}(L)$.  This phenomenon shows $\mathcal{L}_{L}(t,s)$ loses some information of $\mathcal{L}(L)$.

\begin{proposition}
$L$ is an $n$-component classical link with $n\geq 2$. Then, $\mathcal{L}(L)=lk(L)(\mathcal{U}^{n-1}-\mathcal{U}^{n+1})$. Here $\mathcal{U}^{k}$ is a k-component trivial flat link. $lk(L)$ is the sum of writhes of all linking crossing points. Nevertheless, $\mathcal{L}_{L}(t,s)=0$.

\end{proposition}

\begin{proof}

Nontrivial classical flat link does not exist because crossing change is an unlinking and unknotting operation. Hence, $\widetilde{L}_{c}=\mathcal{U}^{n-1}$ if $c$ is a linking crossing point and $\widetilde{L}\cup \mathcal{U}=\mathcal{U}^{n+1}$. So, we have $\mathcal{L}(L)=lk(L)(\mathcal{U}^{n-1}-\mathcal{U}^{n+1})$. On the other hand, $\overline{W}_{\mathcal{U}^{n-1}}(s)=\overline{W}_{\mathcal{U}^{n+1}}(s)=0$. So, $\mathcal{L}_{L}(t,s)=0$.

\end{proof}

Before I use examples to test our theory, let me construct another polynomial invariant $B_{L}(s,t)$. This invariant will be use in flat virtual knot theory.


Notice that a mixture of chord indices is still a chord index.

\begin{proposition}
$weight_{c}(s)= Ind(c)\overline{W}_{L_{c}}(s)$ is a chord index. We require $c$ to be a self crossing point because $Ind$ has no definition on linking crossing points.

\end{proposition}\label{newind}
\begin{proof}
If $c$ is a crossing point involved in $\Omega_{1}$, then $Ind(c)=0$. So, $weight(c)(s)=0$ in this case. Hence, it satisfies (1). There is no difficulty to see that $\widetilde{L}_{c}$ is still a chord index  when we require $c$ to be a self crossing point (Compare the proof of Theorem 3.9 in \cite{Z.Cheng.2018}).  Hence, $\overline{W}_{L_{c}}(s)$ is still a chord index. So, The rest of axioms can be easily verified from the fact that both $Ind(c)$ and $\overline{W}_{L_{c}}(s)$ are chord indices.

\end{proof}

Thus we have following theorem.
\begin{theorem}
$$B_{L}(t,s)=\sum_{c\in S(L)}\omega(c)(t^{weight_{c}(s)}-1)$$ is an invariant of virtual links.

\end{theorem}

In the spirit of Remark \ref{flatrk}, we can use $B_{L}(s,t)$ to construct a flat invariant as follows.

\begin{theorem}
$\overline{B}_{L}(t,s)=B_{L}(t,s)+B_{\overline{L}}(t,s)=B_{L}(t,s)-B_{L}(t^{-1},s)$ is a flat invariant of  virtual links.

\end{theorem}

\begin{proof}
Let $L'$ be the link diagram obtained by switching some crossing point $c$ on $L$. For a crossing point $a$ on $L$, the corresponding crossing point on $L'$ is denoted by $a'$.

If $c$ is a linking crossing point, $\overline{B}_{L}(t,s)=\overline{B}_{L'}(t,s)$ can be seen from the definition of ${B}_{L}(t,s)$.

We assume both $c$ and $a$ are self crossing points. Then $$  \left\{ \begin{aligned}  Ind(a')=Ind(a)\ \text{and}\ \omega(a')=\omega(a)  & &\text{if}\  a\neq c  \\
           Ind(a')=-Ind(a)\ \text{and}\ \omega(a')=-\omega(a) & &\text{if}\  a=c
           \\ \end{aligned}   \right.  $$ and $\overline{W}_{L_{a'}}(s)=\overline{W}_{L_{a}}(s)$ for every self crossing point $a$ on $L$. If $a\neq c$, it is easy to see that $a'$ and $a$ have the same contributions to $\overline{B}_{L'}(t,s)$ and $\overline{B}_{L'}(t,s)$ respectively.  The contribution of $c$ to $\overline{B}_{L}$ is $\omega(c)t^{weight_{c}(s)}-\omega(c)t^{-weight_{c}(s)}=-\omega(c')t^{-weight_{c'}(s)}+\omega(c')t^{weight_{c'}(s)}$, which is also the contribution of $c'$ to $\overline{B}_{L'}(t,s)$.

\end{proof}

\begin{rk}
By Remark \ref{flatrk}, a flat invariant constructed from $\mathcal{L}_{L}(t,s)$ can be $\overline{\mathcal{L}}_{L}(t,s)=\mathcal{L}_{L}(t,s)+\mathcal{L}_{\overline{L}}(t,s)=0$. We get nothing from it. This is the reason of introducing $B_{L}(s,t)$.

\end{rk}

Parallel to properties of $W_{L}$, we give simple and useful properties of ${B}_{L}(t,s)$ and $\overline{B}_{L}(t,s)$.
\begin{proposition}
(a) ${B}_{L}(t,s)=\overline{B}_{L}(t,s)=0$ for classical (flat) links $L$.

(b) If ${B}_{L}(t,s)=\sum_{i=1}^{m}a_{i}t^{g_{i}(s)}$, then the self crossing number of $L$ is no less than $\sum_{i}|a_{i}|$. Especially, ${B}_{K}(t,s)$ gives a lower bound for crossing number of virtual knot $K$.

\end{proposition}
\begin{proof}
$Ind(c)=0$ for every self crossing point of classical link $L$. Hence, $weight(c)=0$. Then ${B}_{L}(t,s)=0$.

Nontrivial classical flat link does not exist. Hence, $\overline{B}_{L}(t,s)=0$ when $L$ is classical.

(b) can be seen from the definition of ${B}_{L}(t,s)$.

\end{proof}

Having these invariants in hand, we can use examples to test them.

\begin{eg}\label{eg2}
\begin{figure}
\centering
 \subfigure{\includegraphics[height = 2cm, width = 3cm]{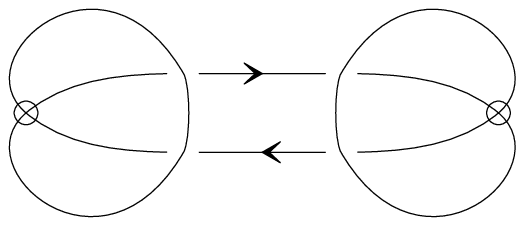}}
\ \
 \subfigure{\includegraphics[height = 2cm, width = 2cm]{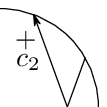}} \caption{Kishino knot}\label{Kishinoknot}

\end{figure}

\begin{figure}
 \centering
 \includegraphics[height =2.5cm, width =10cm]{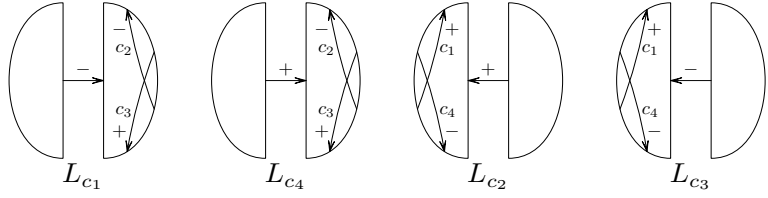}
 \caption{1-smoothing of Kishino}\label{sok}
\end{figure}

Kishino knot $K$ and its Gauss diagram are depicted in Figure \ref{Kishinoknot}. Dye-Kauffman used arrow polynomial to show this knot is nonclassical in \cite{Dye}.  It shows the connected sum of virtual knots is not well-defined.

We have $\mathcal{L}(K)=\widetilde{L}_{c_{1}}-\widetilde{L}_{c_{2}}+\widetilde{L}_{c_{3}}-\widetilde{L}_{c_{4}}$, where $L_{c_{i}}$ ($i\in \{1, 2, 3, 4\}$) is depicted in Figure \ref{sok}.

Through direct computations, we have following equations: $\overline{W}_{L_{c_{1}}}(s)=s^{-1}+s-1-s^{-2}$, $\overline{W}_{L_{c_{2}}}(s)=s^{-1}+s-1-s^{2}$, $\overline{W}_{L_{c_{3}}}(s)=s^{-1}+s-1-s^{-2}$, $\overline{W}_{L_{c_{4}}}(s)=s^{-1}+s-1-s^{2}$. To simplify notations, we set $f(s)= s^{-1}+s-1-s^{2}$. Then $\mathcal{L}_{K}(t,s)=2t^{f(s)}-2t^{f(s^{-1})}$. Hence, Kishino knot is not classical and its real crossing number is 4 by Remark \ref{Kishinork}.

\begin{figure}
 \centering
 \includegraphics[height =2cm, width =3cm]{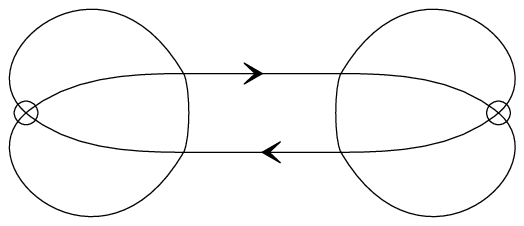}
 \caption{flat Kishino knot}\label{flatKishino}
\end{figure}

The shadow of Kishino knot is depicted in Figure \ref{flatKishino}. we have following data: $Ind(c_{i})=-1$ for $i \in\{1,2,3,4\}$. Then $B_{K}(t,s)=2t^{-f(s)}-2t^{-f(s^{-1})}$, $\overline{B}_{K}(t,s)=2t^{-f(s)}-2t^{-f(s^{-1})}-2t^{f(s)}+2t^{f(s^{-1})}\neq 0$. Hence flat Kishino knot is nontrivial.

\end{eg}


\begin{figure}
\centering
 \subfigure{\includegraphics[height = 2cm, width = 3cm]{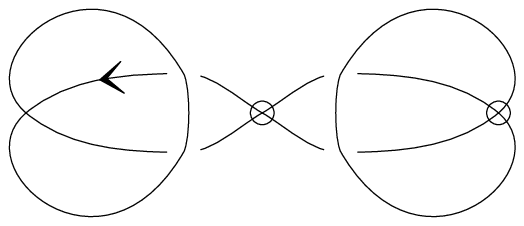}}
\ \
 \subfigure{\includegraphics[height = 2cm, width = 2cm]{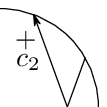}} \caption{a variant of Kishino}\label{variantKishino}

\end{figure}

\begin{figure}
 \centering
 \includegraphics[height =2.5cm, width =10cm]{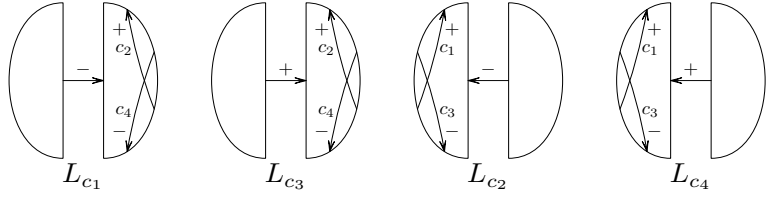}
 \caption{smoothings of $K'$}\label{smoothingofK'}
\end{figure}

\begin{eg}
Let $K'$ denote a variant of Kishino knot, which is depicted in Figure \ref{variantKishino}. We still use $K$ to denote Kishino knot. The results of 1-smoothing of $K'$ are depicted in Figure \ref{smoothingofK'}.

Through direct computations, we have $\mathcal{L}_{K'}(t,s)=t^{f(s)}+t^{-f(s)}-t^{f(s^{-1})}-t^{-f(s^{-1})}$, and $B_{K}(t,s)=2t^{f(s)}-2t^{f(s^{-1})}$. Here $f(s)$ is the same polynomial as in Example \ref{eg2}. Hence, $\overline{B}_{K'}(t,s)=2t^{f(s)}-2t^{f(s^{-1})}-2t^{-f(s)}+2t^{-f(s^{-1})}=-\overline{B}_{K}(t,s)\neq 0$. So, we have the following consequences.

$F(K')$ is nontrivial and different from $F(K)$. So, $K'$ is also nontrivial and different from $K$. And the real crossing number of $K'$ is 4.

\end{eg}

Mellor proved that writhe polynomial $W_{K}(t)$ of a virtual knot is determined by Alexander polynomial $\Delta_{0}(K)(u,v)$ in \cite{Mellor}. Through direct computation we have $\Delta_{0}(K')(u,v)=0$, where $K'$ is the variant of Kishino discussed above. Hence, $\mathcal{L}(K)$ can not be determined by Alexander polynomial.

Next, we use Slavik's knot to point out a defect of our polynomials. This knot appeared in \cite{Dye}.

 \begin{eg}
 \begin{figure}
 \centering
 \includegraphics[height =3cm, width =3cm]{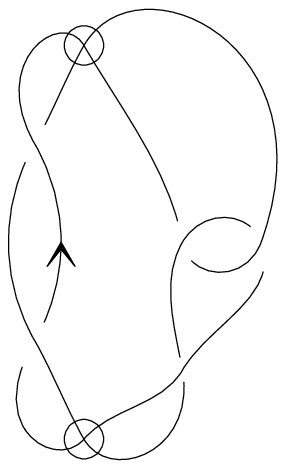}
 \caption{Slavik's knot}\label{eg3}
\end{figure}

 Slavik's knot $K$ is depicted in Figure \ref{eg3}. $W_{K}(t)=\mathcal{L}(K)=\mathcal{L}_{K}(t,s)=B_{K}(t,s)=\overline{B}_{K}(t,s)=0$. The virtuality of $K$ is not detected by arrow polynomial (see \cite{Dye}).

 \end{eg}


The construction of $\mathcal{L}(K)$ used the first layer ideas of arrow polynomial in \cite{Dye}. It seems there are some relations between these three polynomials of this section and arrow polynomial. So, I raise a question as follows.

\textbf{Q:} Are the three polynomials of this section determined by arrow polynomial? More generally, is $\mathcal{L}(L)$ determined by arrow polynomial?

\section{Proof of Lemma \ref{wci}}\label{tech}

Proof of Lemma \ref{wci}:
\begin{proof}

(4) and part of (5) are automatically satisfied because Gauss diagrams are unchanged under $\Omega'_{i}$ ($i=1,2,3$) and $\Omega_{3}^{v}$. In this proof, I always use $d$ to denote an arbitrary self crossing point not involved in the Reidemeister move.

By Lemma \ref{Polyak}, it suffices to check (5) on the four generators of Reidemeister moves. Let $c'$ denote the crossing point involved in $\Omega_{1}$. It easy to see either both endpoints of $c'$ are in $left(d)$ or neither of them are in $left(d)$. Hence, $c'$ has no contribution to $Ind'(d)$. The remaining cases of (5) are on $\Omega_{2a}$ and $\Omega_{3a}$.


For (2'), there are two cases according to $\Omega_{2a}$ is on one or two components.
The Gauss diagrams of $\Omega_{2a}$ in these two cases are shown in Figure \ref{2asingle} and Figure \ref{2atwo} respectively. It is easy to see the contributions of $c_{1}$ and $c_{2}$ to $Ind'(d)$ cancel  out. $Ind'(c_{1})=Ind'(c_{2})$ in these two figures are also obvious. Hence, (2') is satisfied. (5) is satisfied in this case.

\begin{figure}
 \centering
 \includegraphics[height =2cm, width =6cm]{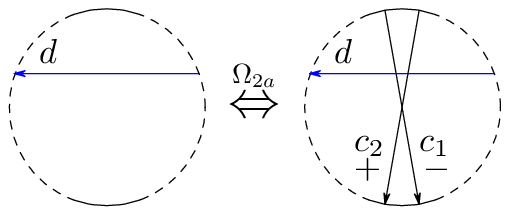}
 \caption{$\Omega_{2a}$ on a single component}\label{2asingle}
\end{figure}

\begin{figure}
 \centering
 \includegraphics[height =2cm, width =6cm]{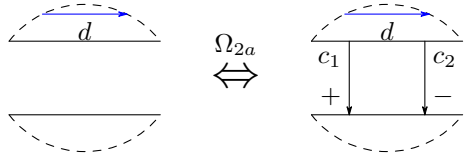}
 \caption{$\Omega_{2a}$ on two components}\label{2atwo}
\end{figure}

For (3'), there are three cases according to $\Omega_{3a}$ is on one, two or three components. We use \Rmnum{1}, \Rmnum{2} and \Rmnum{3} to denote them respectively.

The Gauss diagrams of $\Omega_{3a}$ in \Rmnum{1} are depicted in Figure \ref{3asingle1} and Figure \ref{3asingle2}.

\begin{figure}
 \centering
 \includegraphics[height =2cm, width =6cm]{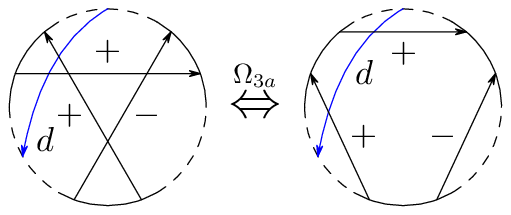}
 \caption{$\Omega_{3a}$ on a single component}\label{3asingle1}
\end{figure}

\begin{figure}
 \centering
 \includegraphics[height =2cm, width =6cm]{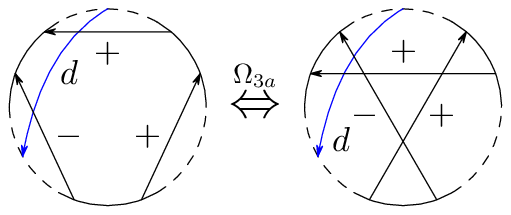}
 \caption{$\Omega_{3a}$ on two components}\label{3asingle2}
\end{figure}

\begin{figure}
\centering
 \subfigure{\includegraphics[height = 2cm, width = 2.5cm]{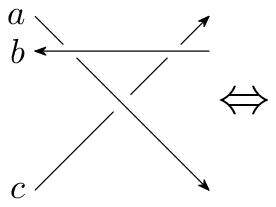}}
\ \
 \subfigure{\includegraphics[height = 2cm, width = 2cm]{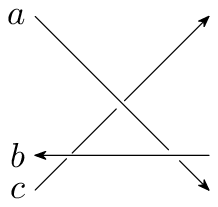}} \caption{$\Omega3a$}\label{3alink}

\end{figure}

Let $a$, $b$ and $c$ denote the three arc in $\Omega_{3a}$ shown in Figure \ref{3alink}. Without loss of generality, there are three subcases in \Rmnum{2}. Those are:

\rmnum{1})  $a,b\in K_{1}$ and $c\in K_{2}$;
\rmnum{2})  $b,c\in K_{1}$ and $a\in K_{2}$;
\rmnum{3})  $a,c\in K_{1}$ and $b\in K_{2}$.

The Gauss diagrams of $\Omega_{3a}$ in \rmnum{1}, \rmnum{2}\ and \rmnum{3}\ are depicted in Figure \ref{3atwo1}, Figure \ref{3atwo2} and Figure \ref{3atwo3} respectively.
\begin{figure}
 \centering
 \includegraphics[height =2.5cm, width =7cm]{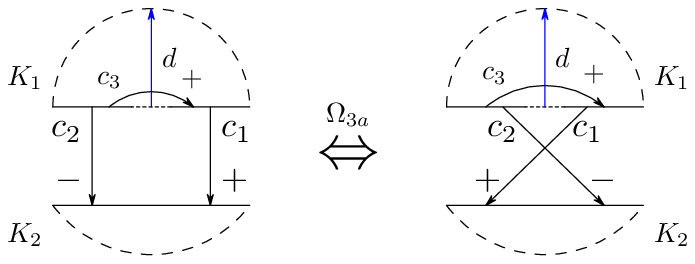}
 \caption{$\Omega_{3a}$ on two components}\label{3atwo1}
\end{figure}
\begin{figure}
 \centering
 \includegraphics[height =2.5cm, width =7cm]{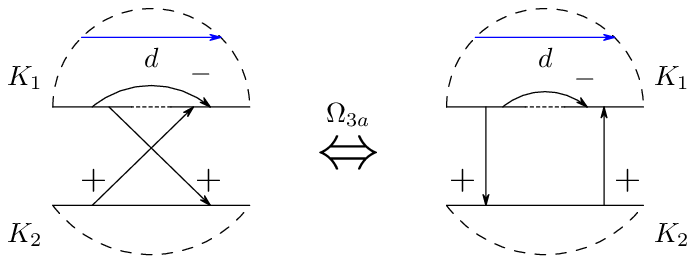}
 \caption{$\Omega_{3a}$ on two components}\label{3atwo2}
\end{figure}
\begin{figure}
 \centering
 \includegraphics[height =2.5cm, width =7cm]{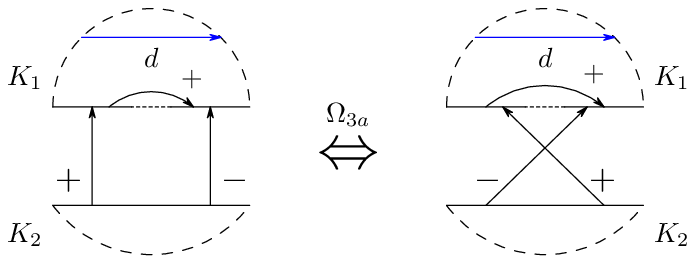}
 \caption{$\Omega_{3a}$ on two components}\label{3atwo3}
\end{figure}

Without loss of generality, we suppose $a\in K_{1}$, $b\in K_{2}$ and $c\in K_{3}$ in case \Rmnum{3}. The Gauss diagrams of $\Omega_{3a}$ in \Rmnum{3} are depicted in Figure \ref{3a3}.
\begin{figure}
 \centering
 \includegraphics[height =3cm, width =10cm]{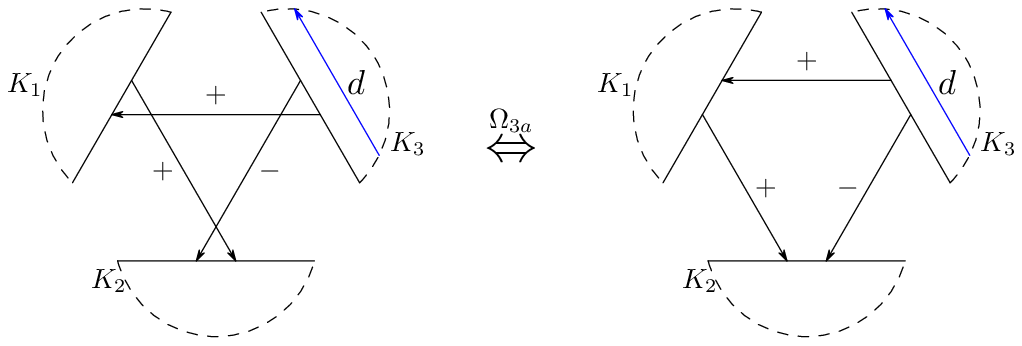}
 \caption{$\Omega_{3a}$ on three components}\label{3a3}
\end{figure}

It is not hard but tedious to directly check (3') and (5) on these Gauss diagrams. We just choose one i.e. Figure \ref{3atwo1} to verify these axioms. The verifications on other diagrams are similar. As shown in this figure, $c_{1}$, $c_{2}$ and $c_{3}$ denote the three crossing points involved in $\Omega_{3a}$. $d$ is an example of a real crossing point not involved in $\Omega_{3a}$. The only self crossing points are $c_{3}$ and $d$. The endpoints of $c_{2}$ are not in $left(d)$. Although $c_{1}$'s and $c_{3}$'s are in $left(d)$, the contributions of $c_{1}$ and $c_{3}$ to $Ind'(d)$ cancel out. Hence $Ind'(d)$ is unchanged in this case. It is easy to see the contributions of $c_{1}$ and $c_{2}$ to $Ind'(c_{3})$ cancel out. $d$ has the same contribution to $Ind'(c_{3})$ on both diagrams. Hence (3') and (5) are satisfied in this case.

In summary, $Ind'$ is a weak chord index.

\end{proof}

There is one thing in above proof I want to point out.  Every component circle is drawn in the unbounded region of the other in the Gauss diagrams of links I have drawn.  You can also draw the circle in the bounded region of the other. These two forms of Gauss diagram are not related by an planar isotopy. But it does not matter because the left parts of self crossing points have nothing to do with the forms.


\begin{ack}
\emph{ The author is grateful to his supervisor Hongzhu Gao for sharing ideas of the weak chord index. The author also wants to thank Zhiyun Cheng for his valuable advice and instructions. The author is supported by NSFC 11771042 and NSFC 11571038.}


\end{ack}


\end{document}